\documentclass[11pt]{amsart}

\usepackage{mathptmx}
\usepackage{helvet}
\usepackage{courier}
\usepackage{graphicx}
\usepackage{multicol}

\usepackage[T1]{fontenc}
\usepackage{euscript}
\usepackage{amsmath}
\usepackage{amsthm}
\usepackage{amssymb}
\usepackage{amscd}
\usepackage{epic}
\usepackage{enumerate}
\usepackage{array}
\usepackage{tabularx}
\usepackage[T1]{fontenc}
\DeclareMathAlphabet{\pazocal}{OMS}{zplm}{m}{n}

\usepackage{amsfonts,amsthm,amsmath}
\usepackage{amssymb}

\numberwithin{equation}{section}
\numberwithin{equation}{subsection}

\theoremstyle{plain}

\newtheorem{theorem}[equation]{Theorem}
\newtheorem{lemma}[equation]{Lemma}

\newtheorem{corollary}[equation]{Corollary}

\newtheorem{clm}[equation]{Claim}

\theoremstyle{definition}

\newtheorem{example}[equation]{Example}
\newtheorem{remark}[equation]{Remark}

\newtheorem{definition}[equation]{Definition}

\newcommand{\calO}{{\mathcal O}}
\newcommand{\cF}{{\mathcal F}}

\newcommand{\C}{{\calc}}

\newcommand{\cG}{\mathcal{G}}

\newcommand{\calI}{\mathcal{I}}

\newcommand{\ord}{{\rm ord}}
\newcommand{\val}{{\rm val}}
\newcommand{\mult}{{\rm mult}}

\newcommand{\Z}{\mathbb{Z}}

\newcommand{\calo}{{\mathcal O}}

\def\C{\mathbb C}

\def\Z{\mathbb Z}

\newcommand{\hookuparrow}{\mathrel{\rotatebox[origin=c]{90}{$\hookrightarrow$}}}

\author{Andr\'as N\'emethi}
\thanks{The first author is partially supported by  ``\'Elvonal (Frontier)'' Grant KKP 144148}
\address{Alfr\'ed R\'enyi Institute of Math.,
Re\'altanoda utca 13-15, H-1053, Budapest, Hungary \newline
 \hspace*{3mm} ELTE - Univ. of Budapest, Dept. of Geo.,
 P\'azm\'any P\'eter s\'et\'any 1/A, 1117, Budapest, Hungary \newline \hspace*{3mm}
  BBU - Babe\c{s}-Bolyai Univ., Str, M. Kog\u{a}lniceanu 1, 400084 Cluj-Napoca, Romania
 \newline \hspace*{3mm}
BCAM - Basque Center for Applied Math.,
Mazarredo, 14 E48009 Bilbao, Basque Country, Spain}
\email{nemethi.andras@renyi.hu }

\author{Willem Veys}
\thanks{The second author is partially supported by KU Leuven grant GYN-E4282-C16/23/010}
\address{KU Leuven, Departement Wiskunde, Celestijnenlaan 200B, 3001 Leuven, Belgium}
\email{wim.veys@kuleuven.be}

\title{Filtrations associated with singularities}

\begin{document}


\begin{abstract} We fix a complex analytic  normal singularity germ $(X,o)$ of dimension $\geq 2$
and a (not necessarily irreducible)
  reduced Weil divisor $(S,o)\subset (X,o)$.
The embedded resolution of the pair determines  a multi-index  filtration of the local ring $\calO_{X,o}$,
 which measures the embedded geometry of the pair. Furthermore, from the (induced) resolution of $(S,o)$ we also consider
  a multi-index  filtration  associated with $(S,o)$. This latter one can be
 lifted to a filtration of  $\calO_{X,o}$  too. The main result proves that the second
 filtration of
  $\calO_{X,o}$ can be realized as a `limit' filtration of the first one (if we blow up certain centers sufficiently many times).
\end{abstract}

\maketitle

\linespread{1.2}

Filtrations of local algebras and their Poincar\'e series  constitute a very important  tool in the study of (local) rings and (local) singularity theory.
Already in low dimensional cases the theory is far from easy. For curves such
a filtration appears as follows. If $(S,o)$
is an irreducible complex analytic  curve singularity and $n:\widetilde{S}\to S$ is its normalization,  then the local ring $\calo_{S,o}$  has a natural valuative filtration given by the
inclusion  $\calo_{S,o}\subset \calo_{\widetilde{S},o}=\mathbb{C}\{t\}$  and
valuation $f\mapsto {\rm ord}_t (n^*f)$. Already when $S$ is a plane curve, the corresponding Poincar\'e series of the filtration of $\calo_{S,o}$ is a crucial invariant,
e.g. it determines the embedded topological type (or, the equisingularity type) of
the pair $(S,o)\subset (\mathbb{C}^2,o)$, cf.  \cite{CDG}, see also \cite{DGN,GorNem2015,NBook}.
For non-irreducible (but reduced) curves we can consider the valuations associated with all the components. The corresponding multivariable Poincar\'e series
is more complicated, its combinatorics is still not completely understood. (For certain topological connections, see \cite{CDG,DGN,Gorsky,GorNem2015,NBook,N_f}.)
There are few results in arbitrary dimension, we mention e.g. \cite{L1,L2,VV}.

Already in the case of the isolated  plane curves  we can notice the parallelism
with the embedded geometry --- what will be our main target.
For instance,  for plane curves, the valuative filtration associated with the normalization can be recovered from the embedded geometry as well.
Indeed,
let $h_1,\dots,h_r$ be irreducible elements in the local ring $\calO_{\C^2,o}$ at the origin of the complex plane (all different up to invertible elements of the local ring), and put $h=\prod_{j=1}^r h_j$. They determine the plane branches $S_1,\dots, S_r$ and the reduced plane curve germ $S$, respectively.
Each $S_j$ determines a (semi)valuation $v_{S_j}$ on $\calO_{\C^2,o}$, given by  $v_{S_j}(f) := \mult_{o} ({\rm div}(f), S_j) = \dim_\C  \calO_{\C^2,o}/(f,h_j)
\in \Z_{\geq 0} \cup \{+\infty\}$ for $f\in \calO_{\C^2,o}$, where $\mult_o (\cdot,\cdot)$ denotes intersection multiplicity, see Subsection \ref{divval}.
(Thus $v_{S_j}(f) = +\infty$ if and only if $h_j$ divides $f$.)
In fact,
we use the same notation for the induced (semi)valuation on  the local ring $\calO_{S,o} = \calO_{\C^2,o}/(h)$, given by $v_{S_j}(\bar{f}):= v_{S_j}(f)$, where $\bar{f}$ is the image of $f\in \calO_{\C^2,o}$ in $\calO_{S,o}$.
The point is that for plane curves the above `curve valuation' and the valuation given by normalization coincide.

For singularities in higher dimensions, one usually introduces filtrations via  divisorial valuations or by the specific forms of the equation of the singularity
(this second case can be exemplified e.g. by
singularities associated with fixed Newton diagrams). In this note
we focus on the first general case  of divisorial filtrations: any fixed irreducible exceptional divisor  of a resolution provides a valuation and a filtration.

One of the important problems here is to understand the relative case, when one tries to compare two natural filtrations. Let us exemplify this by the case of embedded curves.
Let us fix a normal surface singularity $(X,o)$ and an irreducible  Weil divisor $(S,o)\subset (X,o)$ on it.
Then we can consider an embedded resolution $\widetilde{X}$ of the pair $(S,o)\subset (X,o)$ and the irreducible exceptional curve $E_S$ which intersects the strict transform
$\widetilde{S}$ of $S$. Then $E_S$ provides a valuation and filtration on $\calo_{X,o}$. The challenge is to compare this with the valuative filtration of $\calO_{S,o}$ (given by its normalization).
In fact, this second filtration can be lifted to a filtration of $\calo_{X,o}$ via the
projection $\calO_{X,o}\to \calo_{S,o}$.
Usually the two filtrations of $\calO_{X,o}$ do not coincide. However, if we blow up
$\widetilde{X}$ sufficiently many times, first in the point $E_S\cap \widetilde{S}$, and further always in the new intersection point of $\widetilde{S}$ with the exceptional locus, the filtrations associated with these embedded resolutions have a limit, which will coincide with  the lifted filtration obtained from the filtration of  $\calO_{S,o}$. This is in some sense remarkable;
this identification connects two types of geometry:
the abstract analytic geometry of the curve $(S,o)$ and the embedded geometry of the pair
$(S,o)\subset (X,o)$.
Such a correspondence in this low dimensional setting was already noticed in \cite{N},
for plane curves see also \cite{DGN}.\\

 In this article we target the most general setting:

$\bullet$ \  we will work in the local complex analytic category,

$\bullet$ \  the ambient space $(X,o)$  is any normal singularity of dimension $\geq 2$,

$\bullet$ \  the embedded subspace $(S,o)$  is an  arbitrary, not necessarily irreducible,
  reduced Weil divisor.\\

 For the details of the setup and the precise  definition of the filtrations, see Section \ref{result}.
 The main result is formulated in Theorem \ref{th:1}. It identifies the limit filtration, associated with the embedded situation $(S,o)\subset (X,o)$, with the abstract filtration of the Weil divisor $(S,o)$.
 As a consequence, we have the identification of the corresponding Poincar\'e series as well,
 a fact which is exemplified in Section \ref{cor}.  We provide some examples in Section \ref{s:ex}.

 We emphasize that the result is new already in the case of a smooth ambient space $(X,o)=(\mathbb{C}^n,o)$ and irreducible $(S,o)$.

\section{Preliminaries}\label{prelim}

We recall some notions about divisorial valuations and intersection numbers, in particular in the analytic setting.
Note that in the analytic setting there is no analogue of the algebro-geometric notion of generic point of an irreducible variety, hence no \lq global\rq\ instance of a local ring of a variety along a positive dimensional subvariety. 

\subsection{Divisorial valuations}\label{divval}
 Let $M$ be a connected analytic manifold and $V$ an irreducible analytic hypersurface in $M$. For  a point  $P\in V$, let $h$ be a local defining function for $V$ in a neighbourhood of $P$.
For any nonzero analytic function $g$ on $M$, defined near $P$, the {\em (vanishing) order of $g$ along $V$ at $P$}, denoted $\ord_{V,P}(g)$, is the integer $a$ such that we can write $g$  in the (regular) local ring  $\calO_{M,P}$ as $g=h^af$,  with $f$ coprime to $h$.

Since coprime elements of $\calO_{M,P}$ stay coprime in nearby local rings (see e.g. \cite[Proposition page 10]{GH}), we have that $\ord_{V,P}(g)$ is independent of $P$.
Hence we can define $\val_{V}(g)$, the {\em valuation  of $g$ along $V$}, as  $\ord_{V,P}(g)$, for any $P \in V$.  (We put $\val_{V}(0):= +\infty$.)

In the algebraic category this is just the classical divisorial valuation on the function field of the variety $M$ induced by the prime divisor $V$.

\subsection{Intersection multiplicities}\label{intmult}
Let $M$ be an $n$-dimensional connected analytic manifold and $V, W$  analytic hypersurfaces in $M$. For an $(n-2)$-dimensional irreducible subvariety $Z$ of $V\cap W$, we denote by $\mult_Z(V,W)$ the {\em  intersection multiplicity of\, $V$ and $W$ along $Z$}; see e.g. \cite{GH} or \cite{D}. 

In particular, if $n=2$ and thus $Z$ is a point, and the germs of $V$ and $W$ at $Z$ are given as the divisors of $f$ and $g$ in  $\calO_{M,Z}$, respectively, then $\mult_Z(V,W)= \dim_\C  \calO_{M,Z}/(f,g)$.

In general, for all points $P\in Z$ outside a proper analytic subset of $Z$ and any two dimensional submanifold $H$ of $M$ intersecting $Z$ transversally at $P$, we have that the intersection multiplicity $\mult_P(V\cap H, W\cap H)$ (considered in $H$) does not depend on $P$ or $H$. Then $\mult_Z(V,W)$ is equal to this number.
 In the algebraic category this is   $\dim _{K} \calO_{M,Z}/(f,g)$, where $\calO_{M,Z}$ is the local ring of $M$ at (the generic point of) $Z$ with residue field $K$, and the germs of $V$ and $W$ at $Z$ are given as the divisors of $f$ and $g$ in  $\calO_{M,Z}$, respectively.

Note that $\mult_Z(V,W) \in \Z_{\geq 0} \cup \{+\infty\}$, with $\mult_Z(V,W)=+\infty$ if and only if $V$ and $W$ have a common component containing $Z$.


\section{The main result}\label{result}

\subsection{The setup} We fix a germ $(X,o)$ of a normal complex analytic singularity
of dimension $n\geq 2$ and we consider
a reduced (Weil) divisor $(S, o)\subset (X,o)$, with irreducible components (i.e., prime divisors) $S_1,\dots,S_r$.

An important special case is when $(X,o)$ is a complex manifold. Then each $S_j$ is the zero set of an irreducible $h_j$ in the (regular) local ring $\calO_{X,o}$, and $S$ is the zero set of $h=\prod_{j=1}^r h_j$.

\smallskip
Next, we fix an embedded resolution $\pi_0:X_0\to X$ of the pair $(S,o)\subset (X,o)$. That is, $X_0$ is an analytic manifold, $\pi_0$ is a bimeromorphic morphism, all irreducible components $E_i, i\in J,$ of its exceptional locus $E$ and all strict transforms $\widetilde{S_j}, 1\leq j\leq r,$  are codimension one submanifolds of $X_0$, such that $(\cup_{i\in J}E_i)\cup (\cup_{j=1}^r \widetilde{S_j})$ is a simple normal crossing divisor. Note that the strict transform $\widetilde{S}$  of $S$ is $\cup_{j=1}^r \widetilde{S_j}$.
$$ \begin{array}{ccc}
X_0 & \stackrel{\pi_0}{\longrightarrow} & X \\
\hookuparrow & &\hookuparrow\\
\widetilde{S} & \stackrel{p}{\longrightarrow}& S\\
\hookuparrow & &\hookuparrow\\
\widetilde{S_j} & \stackrel{p_j}{\longrightarrow}& S_j

 \end{array}$$
We may suppose $\pi_0$ is such that for all $E_i$ we have that $E_i \cap \widetilde{S}$ is either empty or irreducible; note that this implies that all $\widetilde{S_j}$ are disjoint.
We set $C_i:=E_i\cap \widetilde{S}$ whenever this intersection is nonempty, say for $i\in I\subset J$. (So then $C_i=E_i\cap \widetilde{S_j}$ for exactly one $j$.)

 \begin{remark}

\begin{enumerate}
\item The $r$ connected components of $\cup_{i\in I} C_i$ are precisely the intersections of the exceptional locus with a fixed $\widetilde{S_j}$.  In fact, the restriction  $p_j=\pi_0|_{\widetilde{S_j}}$ is a log resolution of $S_j$, with exceptional locus such a connected component $E\cap \widetilde{S_j}$.

\item  Evidently, the collection $\{C_i\}_{i\in I}$ depends on the choice of $\pi_0$ (except for $n=2$, when the curves $E_i$ depend on the choice of $\pi_0$, but the points $C_i$ not).

 \item We can restrict the construction below to some subset $\{C_i\}_{i\in I'}$, for certain
  $I'\subset I$, which is (for example) intrinsically associated with $p$, or even to $S$ itself.
 A typical example is the case $n=3$ and $S$ irreducible, where the curves $\{C_i\}_{i\in I'}$ are the (irreducible)
 exceptional curves in the minimal good resolution of $S$, or those exceptional curves which appear in any fixed embedded resolution of $S\subset X$.
\end{enumerate}
 \end{remark}

 Next, we proceed with the following construction.
 We blow up in $X_0$ the  centers $\{C_i\}_{i\in I}$  one by one, and in this way we obtain $X_1$ with
 new exceptional components $\{E_{i,1}\}_{i\in I}$ and strict transforms (still denoted by)
 $\widetilde{S}$ and $\widetilde{S_j}$. We will use the same notation $C_i$ for the intersection
 $\widetilde{S}\cap E_{i,1}$ as well.

We  continue further in this way: we blow up $X_1$ along the centers $\{C_i\}_{i\in I}$, and we obtain $X_2$ with new exceptional components $\{E_{i,2}\}_{i\in I}$, etc. That is, for any $m\in \Z_{> 0}$, we obtain after $m$ blow-ups the manifold $X_m$  with new exceptional divisors
  $\{E_{i,m}\}_{i\in I}$ and  intersections
 $C_i=\widetilde{S}\cap E_{i,m}$ on $X_m$.
 We write $\pi_m$ for the modification $X_m\to X$.
 $$ \begin{array}{cccccccccc}
\longrightarrow & X_m & \longrightarrow & \ldots   & \longrightarrow & X_1  & \longrightarrow &
X_0 & \stackrel{\pi_0}{\longrightarrow} & X \\ &
\hookuparrow & & & &\hookuparrow & & \hookuparrow &  &\hookuparrow  \\
\stackrel{=}{\longrightarrow} &
\widetilde{S} & \stackrel{=}{\longrightarrow} & \ldots &  \stackrel{=}{\longrightarrow} &
\widetilde{S} & \stackrel{=}{\longrightarrow} &
\widetilde{S} & \stackrel{p}{\longrightarrow}& S \\&
\hookuparrow & & & &\hookuparrow & & \hookuparrow &  &\hookuparrow  \\
\stackrel{=}{\longrightarrow} &
\widetilde{S_j} & \stackrel{=}{\longrightarrow} & \ldots &  \stackrel{=}{\longrightarrow} &
\widetilde{S_j} & \stackrel{=}{\longrightarrow} &
\widetilde{S_j} & \stackrel{p_j}{\longrightarrow}& S_j
\end{array}$$

\begin{remark} (1)
When $n \geq 3$, the spaces $X_1, \ldots , X_m, \ldots $,
and also the (isomorphism class of the) $E_{i,m}$ depend on the order in which we blow up the centers $\{C_i\}_{i\in I}$ in $X_m$, at each step.
However, the divisorial valuations $\val_{E_{i,m}}$ do not depend on this order.

(2)
Note that in this analytis setup
the centres of blow-up $C_i$ and exceptional components can be not projective (and not algebraic).

(3) On the other hand, we can replace our analytic setup, and consider all the constructions in the
(complex) algebraic category.
When  $(S, o)\subset (X,o)$ are algebraic, we assume that $X_0$ and $\pi_0$ are algebraic.
\end{remark}

\subsection{The filtrations}
Let $\calI(S)$ denote the ideal of $S$ in $\calO_{X,o}$. We denote by  $\bar{f}\in \calO_{S,o}=\calO_{X,o}/\calI(S)$  the residue class of $f\in \calO_{X,o}$.
For any $\bar{f}\in\calO_{S,o}$, we can consider its valuation $\val_{C_i}(\bar{f})$.  That is, we consider (the pullback of) $\bar{f}$ as function on $\widetilde{S_j}\subset \widetilde{S}$, where  $C_i:=E_i\cap \widetilde{S_j}$, and then $\val_{C_i}(\bar{f})$ is the valuation of $\bar{f}$ along $C_i$ (considered on the manifold $\widetilde{S_j}$) as in Subsection \ref{divval}.

Note that $\val_{C_i}(\bar{f})=+\infty$ if and only if $f$ vanishes on $S_j$.  When $S=S_j$ is irreducible, this means just that $\bar{f}=0$ in $\calO_{S,o}$.
(So in general $\val_{C_i}$ is only a semivaluation on $\calO_{S,o}$; it is a valuation if and only $S$ is irreducible.)

Similarly, for any $f\in\calO_{X,o}$, we can consider its valuation $\val_{E_{i,m}}(f)$.

\medskip
We consider the following multi-index filtrations on $\calO_{S,o}$ and $\calO_{X,o}$.
For  $\ell= (\ell_i)_{i\in I}$ ranging over $ (\Z_{\geq 0})^{|I|}$, we define on $\calO_{S,o} $ the (decreasing) filtration
$$\overline{\cG_I}(\ell):= \{ \bar{f}\in \calO_{S,o}\ :\, \val_{C_i}(\bar{f})\geq \ell_i \ \ \mbox{for
all $i\in I$}\ \}.$$
Its lifting to $\calO_{X,o}$ is the filtration
$$\cG_I(\ell):= \{ f\in \calO_{X,o}\ :\, \val_{C_i}(\bar{f})\geq \ell_i \ \ \mbox{for
all $i\in I$}\ \}.$$


\noindent
Similarly, for any $m\geq 0$, we define on $\calO_{X,o} $ the decreasing filtration
$$\cF_{I,m}(\ell):= \{ f\in \calO_{X,o}\ :\, {\rm val} _{E_{i,m}}(f)\geq \ell_i \ \ \mbox{for
all $i\in I$}\ \}.$$
Note that one has the natural inclusions of filtrations $\cF_{I,m}\subset \cF_{I, m+1}$ (that is,
 $\cF_{I,m}(\ell)\subset \cF_{I, m+1}(\ell)$ for any $m$ and $\ell$).

\subsection{Limit filtrations}

The main result of this article is the following statement.
It is a generalization of \cite[Corollary 2]{DGN} and \cite[Section 9]{N}
to the present very general setting.

\begin{theorem}\label{th:1}
(1) For each $\ell \in (\Z_{\geq 0})^{|I|}$  there exists $M\in \Z_{\geq 0}$ such that
$$\cF_{I,m}(\ell)=\cF_{I, M}(\ell) \ \ \mbox{for any $m\geq M$}. $$
In particular, the limit filtration $\lim_m \cF_{I,m}$, given by
$(\lim_m \cF_{I,m})(\ell):= \lim_{m\to \infty} (\cF_{I,m}(\ell))$, is a well defined filtration on $\calO_{X,o}$.

(2) We have the equality of filtrations
 $$\lim_m \cF_{I,m}=\cG_I.$$
\end{theorem}
\begin{proof}
First, we describe ${\rm val}_{C_i}(\bar{f})$ in terms of the intersection multiplicity $\mult_{C_i}(\ , \ )$ on any $X_m$.
The following lemma is probably essentially known; we give a proof for completeness.

\begin{lemma}\label{lemma} Fix $i\in I$ and say $C_i=E_{i,m}\cap \widetilde{S_j}$ in $X_m$.
We have for any $f \in \calO_{X,o}$  that
$$\val_{C_i}(\bar{f}) = \mult_{C_i} ({\rm div}( \pi^*_mf), \widetilde{S_j}).$$
In particular, the right hand side does not depend on $m$.
\end{lemma}
\begin{proof}
If $f$ vanishes on $S_j$, then both sides are $+\infty$.  So we assume from now on that this is not the case.

Take a point $P$ in $C_i$, and choose coordinates $x,y, z=(z_1, \dots,z_{n-2})$ in a neighbourhood of $P$ in $X_m$, such that, in the local ring $\calO_{X_m,P}$, the zero sets of $x$ and $y$ are  $E_{i,m}$ and $\widetilde{S_j}$, respectively. Say $\pi_m^*(f)=\alpha(x,y,z)$
 in $\calO_{X_m,P}$.

Then we have in the local ring $\calO_{\widetilde{S_j},P}=\calO_{X_m,P}/(y)$ that   $p_j^*(f|_{S_j})=(\pi_m^*(f))|_{\widetilde{S_j}}=\alpha(x,0,z)$. Also, in this local ring $C_i$ is the zero set of $x$.
Then, by definition,  $\val_{C_i}(\bar{f})$ is the unique number $k$ such that we can write $\alpha(x,0,z)=x^k (\beta(z)+x\gamma(x,z))$ in $\calO_{\widetilde{S_j},P}$, with $\beta(z)\neq 0$.

 Let now $H$ be the surface in that neighbourhood given by $z_1=\dots=z_{n-2}=0$.
Then
$$\mult_{P}({\rm div}( \pi^*_mf)\cap H, \widetilde{S_j}\cap H) =
\dim_\C \frac{\calO_{H,P}}{(\alpha(x,y,0),y)}=\dim_\C \frac{\C\{x\}}{\alpha(x,0,0)}=\dim_\C \frac{\C\{x\}}{x^k(\beta(0)+x\gamma(x,0))}.$$
For any point $P_a=(0,0,a_1,\dots,a_{n-2})$ of $C_i$ close to $P$, we can take $(x,y,z_1-a_1,\dots,z_{n-2}-a_{n-2})$ as coordinates around $P_a$.  Then, analogously, with $H_a$ given by $z_1-a_1=\dots=z_{n-2}-a_{n-2}=0$, we have that
$\mult_{P_a}({\rm div}( \pi^*_mf)\cap H_a, \widetilde{S_j}\cap H_a)=k$ outside the proper analytic subset where $\beta(z)=0$. Hence $\mult_{C_i} ({\rm div}( \pi^*_mf), \widetilde{S_j})=k$.

(In the algebraic category, the proof is more conceptual in terms of the local ring of $X_m$ at the generic point of $C_i$.)
\end{proof}


Next, note the following identity of divisors on $X_m$:
$${\rm div}(\pi^*_m\,f)=\val_{E_{i,m}}(f)  E_{i,m}+ {\rm div}(\widetilde{f})+\cdots ,$$
where ${\rm div}(\widetilde{f})$ is the strict transform of ${\rm div}(f)$ and the remaining terms have support not containing $C_i$.  Therefore,
$$\mult_{C_i}({\rm div}( \pi^*_mf), \widetilde{S_j})= \val_{E_{i,m}}(f) + \mult_{C_i}({\rm div}(\widetilde{f}), \widetilde{S_j}) \geq  \val_{E_{i,m}}(f).$$
Hence, by the lemma, we have that
 $$\cF_{I,m}\subset \cG_I \ \ \mbox{for any $m$}.$$
\begin{clm}\label{claim}
For each $\ell \in (\Z_{\geq 0})^{|I|}$, consider $M=M(\ell):= max_{i\in I}  \ell_i$.
Then
$$\cG_I(\ell)\subset \cF_{I,M}(\ell).$$
\end{clm}
\begin{proof}
Take $f\in \cG_I(\ell)$. On $X_M$ we can have two possibilities.

Either ${\rm div}(\widetilde{f})\supset C_i$. Then necessarily also ${\rm div}(\widetilde{f})\supset C_i$ on every $X_k \, (0\leq k < M$), and hence at every blow-up $X_{k+1}\to X_k$ the valuation of $f$ along each $E_{i,k}$ strictly increases. In this case ${\rm val}_{E_{i,M}}(f)\geq M\geq \ell_i$ for all $i$.

Or, ${\rm div}(\widetilde{f})\not\supset C_i$. In this case  $\mult_{C_i}({\rm div}(\widetilde{f}), \widetilde{S_j})=0$, and hence ${\rm val} _{E_{i,M}}(f)=\mult_{C_i}({\rm div}( \pi^*_Mf), \widetilde{S_j})\geq \ell_i$ (by the lemma).
\end{proof}

\vspace{1mm}

This proves part {\it (1)} of the Theorem, and in fact the statement $\lim_m \cF_{I,m}=\cG_I$ too, hence part {\it (2)} is also established.
\end{proof}


\section{Corollaries on Poincar\'e series}\label{cor}
\subsection{}
The filtered algebras can be coded in a multivariable Poincar\'e series, cf. \cite{CDK}, see also e.g.
\cite{CDG,N,NBook}.
For this it is convenient to extend/define the filtrations for any $l\in \mathbb{Z}^{|I|}$
(instead of $\mathbb{Z}_{\geq 0}^{|I|}$). For example,  we extend the filtration
$\mathcal{G}_I$ via $\mathcal{G}_I(l):=\mathcal{G}_I(\max\{(0,0,\dots,0),l\})$. Here for any $l,l'\in\mathbb{Z}^{|I|}$ we set $\max\{l,l'\}:=l''$, where $l''_i=\max\{l_i,l'_i\}$ for every $i$.
For the other filtrations we proceed similarly.

\begin{definition} Denote $\bar{1}:= (1,1,\dots,1)$ in $\mathbb{Z}^{|I|}$.
Assume that all quotients $\cF_{I,m}(\ell)/\cF_{I,m}(\ell+\bar{1})$ and $\overline{\cG_I}(\ell)/\overline{\cG_I}(\ell+\bar{1}) \cong \cG_I(\ell)/\cG_I(\ell+\bar{1}) $ are finite dimensional. Then the Poincar\'e series associated with the filtrations $\cF_{I,m}$ and $\overline{\cG_I}$, say in variables $t=\{t_i\}_{ i\in I}$, are
$$P_{\cF_{I,m}}(t) := \frac{\prod_{i\in I}(t_i -1)}{\prod_{i\in I}t_i -1}\sum_{\ell \in \Z^{|I|}}  \dim \left( \frac{\cF_{I,m}(\ell)}{\cF_{I,m}(\ell+\bar{1})}\right)$$
and
$$P_{\overline{\cG_I}}(t) := \frac{\prod_{i\in I}(t_i -1)}{\prod_{i\in I}t_i -1}\sum_{\ell \in \Z^{|I|}}  \dim \left( \frac{\overline{\cG_{I}}(\ell)}{\overline{\cG_{I}}(\ell+\bar{1})}\right),$$
respectively.
\end{definition}

\begin{corollary} We have that
$$\lim_m P_{\cF_{I,m}}(t) = P_{\overline{\cG_I}}(t).$$
\end{corollary}

This is a vast generalization of \cite[Corollary 2]{DGN} about a reduced plane curve germ $S$, consisting of branches $S_1,\dots, S_r$.  Note that their right hand side is the Poincar\'e series of the curve valuations $v_{S_i}(f)$; but we can consider this as a special case of the setting in Section \ref{result} with $(X,o)= (\C^2,o)$. In this dimension $I=\{1,\dots,r\}$ and the $C_i$ are the points $E_{i,m} \cap \widetilde{S_i}$ for all $i$ and all $m$.
By elementary intersection theory, the curve valuation $v_{S_i}(f) = \mult_{o} ({\rm div}(f), S_i)$
 is equal to $\mult_{C_i} ({\rm div}( \pi^*_mf), \widetilde{S_i})$, and hence  $v_{S_i}= \val_{C_i}$ by Lemma \ref{lemma}.

\section{Examples}\label{s:ex}

\subsection{} Usually the computations of the Poincar\'e series are not easy.
Here we provide one example in dimension $\dim (X,o)=2$, where the computation can be done
by a toric method and also by the  general technique of (splice quotient) normal surface singularities.
The second example is a non-toric 3-dimensional example.

Both examples treat one valuation, i.e., $|I|=1$.
\begin{example}\label{ex:1}
Consider the cyclic quotient singularity $X_{5,2}={\rm Spec}(\mathbb{C}[\mathcal{S}])$, where
$\mathcal {S}$ is the semigroup of integral lattice points in the cone
$\sigma^\vee=\mathbb{R}\langle (1,0), (2,5)\rangle$.
(For details regarding cyclic quotient singularities, see e.g. \cite[2.3]{NBook}, whose notations we will adopt.)

 The lattice points $(1,0),\ (1,1),\ (1,2)$ and $ (2,5)$ correspond to generators of the local algebra, these are denoted by $x,\ y,\ z$ and $w$, respectively. They satisfy the relations
 $xz=y^2, \ xw=yz^2$ and $ yw=z^3$. This $X_{5,2}$ is our singularity $(X,o)$ and we fix on it the Weil divisor $(S,o)$ given by $y=z=w=0$.
The  minimal (embedded) resolution $\widetilde{X}_0$ can be read from the cone $\sigma=
 \mathbb{R}\langle (5,-2), (0,1)\rangle$. The semiline   $\mathbb{R}\langle (1,0)\rangle$
 corresponds to the exceptional divisor $E_1=E_{1,0}$ with self-intersection number $-3$,
 the
  semiline   $\mathbb{R}\langle (3,-1)\rangle$
 corresponds to the exceptional divisor $E_2=E_{2,0}$ with self-intersection number $-2$,  and the  semiline   $\mathbb{R}\langle (0,1)\rangle$
 corresponds to the strict transform $\widetilde{S}$ of $S$; it  intersects $E_1$ in the point $C_1$.
 Blowing up $\widetilde{X}_0$ consecutively in $C_1$ corresponds to
 introducing new semilines  in $\sigma$. The new exceptional curve $E_{1,1}$ of the resolution $\widetilde{X}_1$ corresponds to  $\mathbb{R}\langle (1,1)\rangle$, at further  steps $m$ (when we create  $\widetilde{X}_m$) we introduce additionally the
 semiline $\mathbb{R}\langle (1,m)\rangle$. It corresponds to the exceptional curve $E_{1,m}$, satisfying $E_{1,m} \cap \widetilde{S}=\{C_1\}$ in $\widetilde{X}_m$.

 The filtration ${\mathcal F}_{I,m}$ of the local ring  $\mathcal{O}_{X,o}$ is given by the weights of $E_{1,m}$ on the monomials of $(\mathbb{C}[\mathcal{S}])$, namely the weight of $x$ is given by the inner product $((1,m),(1,0))=1$, the weight of $y$ is  $((1,m),(1,1))=1+m$,
  the weight of $z$ is  $((1,m),(1,2))=1+2m$, and the weight of $w$ is  $((1,m),(2,5))=2+5m$.
 Then the Poincar\'e series  associated with the filtration
  ${\mathcal F}_{I,m}$ in $\mathcal{O}_{X,o}$  is the Poincar\'e series associated with this grading  of $\mathbb{C}[\mathcal{S}]$.
  We decompose the lattice points in the cone $\sigma^\vee$ according to their positions in different horizontal lines. Let us consider the first five lines. The monomials sitting right of $(0,0)$ (monomials $1,x,x^2,\ldots$)
  provide $\sum_{k\geq 0}t^k =1/(1-t)$, the monomials sitting right of $(1,1)$ (monomials $y,yx,yx^2,\ldots$)
  provide $\sum_{k\geq 0}t^{1+m} \cdot t^k =t^{1+m}/(1-t)$, the next horizontal line gives  $t^{1+2m}/(1-t)$, then we have
  $t^{2+3m}/(1-t)$ (generated by $(2,3)$) and $t^{2+4m}/(1-t)$ (generated by $(2,4)$.
 Starting from the 6th line, the line of $w$, everything repeats again (by the shift of powers of $t$), hence
  the Poincar\'e series $P_{{\mathcal F}_{I,m}}(t)$ is
  $$P_{{\mathcal F}_{I,m}}(t)=\frac{1+t^{1+m}+t^{1+2m}+t^{2+3m}+t^{2+4m}}{(1-t)(1-t^{2+5m})}.$$
   Its limit is $1/(1-t)$, which is indeed the (intrinsic) Poincar\'e series of the smooth
   curve germ $(S,o)$.

   In fact,   $P_{{\mathcal F}_{I,m}}(t)$ can be computed using multivariable topological Poincar\'e series as well.
   Let us fix the resolution $\widetilde{X}_m$. In this situation we can consider the two--variable
     (analytic) Poincar\'e series $P_o$ associated with the divisorial filtration given by
    the exceptional curves $E_{1,m}$ and $E_{2,m}$.
In this case it
   equals the  two-variable topological Poincar\'e series  $Z_o$ (see e.g. Theorem 8.5.16 in \cite{NBook}).
   This  two--variable  topological Poincar\'e series  (cf. \cite[8.4]{NBook}), with variables
   $(t,s)$ corresponding to $E_{1,m}$ and $E_{2,m}$, respectively, is
   \begin{equation*}\begin{split}Z_o(t,s)&=
   \frac{1}{(1-t^{(2+5m)/5}s^{1/5})(1-t^{1/5}s^{3/5})}\\ &=
   \frac{(1+t^{(2+5m)/5}s^{1/5}+\ldots + t^{4(2+5m)/5}s^{4/5})( 1+t^{1/5}s^{3/5}+\ldots+t^{4/5}s^{12/5}) }{(1-t^{2+5m}s)(1-ts^3)}.\end{split}\end{equation*}
   Then $P_{{\mathcal F}_{I,m}}(t)$ is obtained from $Z_o(t,s)$ in two steps. First, we consider that subseries
    $Z$ which has all monomials with integral exponents, then we substitute $s=1$, finally obtaining the above  $P_{{\mathcal F}_{I,m}}(t)$.

 \begin{example}\label{ex:2}
 Take $(X,o)=({\mathbb {C}}^3,o)$ and let $(S,o)$ be given by the equation
 $xy-z^2=0$. Let the first embedded  resolution ${X}_0$ be the blow-up of $(X,o)$ at the origin. Then $\widetilde{S}$ is the minimal resolution of $(S,o)$ with one exceptional curve $C$, which is
 a rational $(-2)$-curve on $\widetilde{S}$. In our construction we blow up this curve consecutively, obtaining in this way the embedded resolutions $\{{X}_m\}_{m\geq 1}$.

 The Poincar\'e series $P_{\overline{{\mathcal {G}_I}}}(t)$ is easy, it is the Poincar\'e series of the homogeneous degree 2 hypersurface associated with its weights $(1,1,1)$, hence
  $$P_{\overline{{\mathcal {G}_I}}}(t)=\frac{1-t^2}{(1-t)^3}.$$
  The Poincar\'e series  $P_{{\mathcal F}_{I,m}}(t)$ of the filtrations on ${\mathcal O}_{{\mathbb C}^3,0}=\mathbb{C}\{x,y,z\}$ can be obtained as follows.
  Rewrite any local series  $f(x,y,z)$ in a unique way in the form
  $\sum a_{i,j,k,l} x^iy^jz^k(xy-z^2)^l$, where $i\geq 0$, $j\geq 0$, $k\in \{0,1\}$, and $l\geq 0$.
  We will analyse the functions $x,y,z,(xy-z^2)$ 
  along the consecutive steps of the blow-up.

  The first blow-up is given by $x=u$, $y=uv$, $z=uw$ (it is enough to consider only this chart). Then the exceptional surface $E_0\subset X_0$ is given by $u=0$, in which $C$ is given by $u= v-w^2=0$. We perform the change of variables
 $\tilde{v}=v-w^2$; in these new coordinates, we have that
  $x=u$, $y=u(\tilde{v}+w^2)$, $z=uw$ and $xy-z^2=u^2\tilde{v}$. The curve $C$ is now $\{u=\tilde{v}=0\}$.
  Then the next blow-up (in the relevant chart) is given by $u=u_1$, $w=w_1$ and $\tilde{v}=u_1\tilde{v}_1$; hence in this coordinate chart we have   $x=u_1$, $y=u_1(u_1\tilde{v}_1+w_1^2)$, $z=u_1w_1$ and $xy-z^2=u_1^3\tilde{v}_1$. And the curve $C$ is given in $X_1$ by $\{u_1=\tilde{v}_1=0\}$.
  By induction, in the relevant  chart of $X_m$, with coordinates
  $u_{m}$,  $\tilde{v}_{m}$, $w_{m}$, the new exceptional surface $E_m$ and the curve $C$ are  $\{u_m=0\}$ and $\{u_m=\tilde{v}_m=0\}$, respectively. And the
  pullbacks of the functions are
   $x=u_{m}$, $y=u_{m}(u_{m}^m\tilde{v}_{m}+w_{m}^2)$,
   $z=u_{m}w_{m}$ and $xy-z^2=u_{m}^{m+2}\tilde{v}_{m}$.
   Note that then the pullback of  $x^iy^jz^k(xy-z^2)^l$
   is $u_{m}^{i+j+k+(m+2)l} (u_{m}^m\tilde{v}_{m}+w_{m}^2)^j
   w_{m}^k\tilde{v}_{m}^l$.
   In particular, an expression  of type  $x^iy^jz^k(xy-z^2)^l$ has $E_m$-valuation $i+j+k+(2+m)l$.

   For any fixed $\ell\geq 0$, we claim that the classes of the functions   $x^iy^jz^k(xy-z^2)^l$, where $i,j,l\geq 0$ and $k\in\{0,1\}$ and
   $i+j+k+(2+m)l=\ell$, form  a basis of ${\mathcal F_{I,m}}(\ell)/ {\mathcal F_{I,m}}(\ell+1)$.
   First,  we show that their classes  are linearly independent.
    For this, consider a linear combination
    $\sum c_{i,j,k,l} x^iy^jz^k(xy-z^2)^l$, and assume that it belongs to  ${\mathcal F_{I,m}}(\ell+1)$. This means that
    $$\sum c_{i,j,k,l}(u_{m}^m\tilde{v}_{m}+w_{m}^2)^j
   w_{m}^k\tilde{v}_{m}^l \equiv 0 \ ({\rm mod}\ u_{m})$$
  or, equivalently, that
 $$ \sum c_{i,j,k,l}    w_{m}^{2j+k}\tilde{v}_{m}^l    \equiv 0 \ \ (\mbox{in ${\mathbb C}[\tilde{v}_{m},w_{m}]$}).$$
   Here the summation is over $i,j,l\geq 0$, $k\in \{0,1\}$ and $i+j+k+(m+2)l=\ell$, for $\ell$ fixed.
   Since $\tilde{v}_{m}$ and $w_{m}$ are free coordinates, and because  $k\in\{0,1\}$, the expression determines
   uniquely  $j$,$k$ and $l$. And since $\ell$ is fixed, also $i$ is determined. We conclude that there are no cancellations, and hence $c_{i,j,k,l}=0$ for all $i,j,k,l$.

    This computation also shows that the functions   $x^iy^jz^k(xy-z^2)^l$, where $i,j,l\geq 0$ and $k\in\{0,1\}$ and
   $i+j+k+(2+m)l=\ell$, generate  ${\mathcal F_{I,m}}(\ell)/ {\mathcal F_{I,m}}(\ell+1)$.
   Indeed, by a similar argument as above, any combination of type
    $\sum c_{i,j,k,l} x^iy^jz^k(xy-z^2)^l$, with $i+j+k+(m+2)l<\ell $ cannot be an element of
    ${\mathcal F_{I,m}}(\ell)$.

By  this discussion the wished Poincar\'e series
    coincides with the Poincar\'e series of the local algebra $\mathbb{C}\{x,y,z,w\}/(z^2)$,
  weighted so that the weights of $x$, $y$ and $z$ are 1 and the weight of $w$ is $m+2$.
   Therefore, \begin{equation}\label{eq:NEW}P_{{\mathcal F}_{I,m}}(t)=\frac{1-t^2}{(1-t)^3(1-t^{m+2})}.\end{equation}

 \noindent  In a different way,  we can also argue as follows. Consider the exact sequence
  $$0\to {\mathbb C}\{x,y,z\}\stackrel{\cdot (xy-z^2)}{\longrightarrow}
  {\mathbb C}\{x,y,z\}\to {\mathbb C}\{x,y,z\} /(xy-z^2)\to 0.$$
By the arguments above, it induces more precisely an exact sequence of filtrations;
 we have for all $\ell$ the exact sequence
 $$0\to {\mathcal F_{I,m}}(\ell-m-2)\stackrel{\cdot (xy-z^2)}{\longrightarrow}
  {\mathcal F_{I,m}}(\ell)\to ({\mathbb C}\{x,y,z\} /(xy-z^2))(\ell)\to 0,$$
where on the right term we mean the filtration induced by the monomials of type
$x^iy^jz^k$, where $i,j\geq 0$ and $k\in\{0,1\}$ and
   $i+j+k=\ell$. This coincides with the graded  local algebra $\mathbb{C}\{x,y,z\}/(z^2)$,
  weighted so that the weights of $x$, $y$ and $z$ are 1. Thus
  the Poincar\'e series of ${\mathbb C}\{x,y,z\} /(xy-z^2)$ is
  $(1-t^2)/(1-t)^3$. From these exact sequences we thus obtain
 $$ P_{{\mathcal F}_{I,m}}(t)= P_{{\mathcal F}_{I,m}}(t)\cdot t^{m+2}+ (1-t^2)/(1-t)^3,$$
 which gives (\ref{eq:NEW}) again.

  The limit property can be seen in this case as well.

 \end{example}

\end{example}

\end{document}